\documentclass[a4paper,12pt,oneside,reqno]{amsart}
\usepackage{amssymb,amsfonts,amsmath,amsthm,amscd, bbm}
\usepackage{graphicx, color}
\numberwithin{equation}{section}  

\newcommand{\beq}{\begin{equation}} 
\newcommand{\eeq}{\end{equation}} 
\newcommand{\bea}{\begin{aligned}}
\newcommand{\eea}{\end{aligned}}
\newcommand{\bdm}{\begin{displaymath}}
\newcommand{\edm}{\end{displaymath}}
\newcommand{\barr}{\begin{array}}
\newcommand{\earr}{\end{array}}
\newcommand{\ben}{\begin{enumerate}}
\newcommand{\een}{\end{enumerate}}
\newcommand{\bde}{\begin{description}}
\newcommand{\ede}{\end{description}}

\newtheorem{teor}{Theorem}[section]
\newtheorem{prop}[teor]{Proposition} 
\newtheorem{lem}[teor]{Lemma}

\newcommand{\PP}{\mathbb{P}}
\newcommand{\E}{{\mathbb{E}}}

\newcommand{\defi}{\equiv} 

\newcommand{\s}{\sigma}
\newcommand{\vare}{\varepsilon}

\newcommand{\1}{\mathbbm{1}}

\newenvironment{dedication}
        {\vspace{6ex}\begin{quotation}\begin{center}\begin{em}}
        {\par\end{em}\end{center}\end{quotation}}

\begin{document}

\begin{dedication}
\hfill
\vspace*{1cm}{To Anton Bovier, mentor and friend, on the occasion of his $60^{th}$ birthday.}
\end{dedication}
\title[McKean, Bovier and Hartung]
{On McKean's martingale in the \\ Bovier-Hartung extremal process} 

\author[C. Glenz]{Constantin Glenz}            
 \address{Constantin Glenz \\ J.W. Goethe-Universit\"at Frankfurt, Germany.}
 \email{constantin.glenz@gmail.com}

\author[N. Kistler]{Nicola Kistler}
\address{Nicola Kistler \\ J.W. Goethe-Universit\"at Frankfurt, Germany.}
\email{kistler@math.uni-frankfurt.de}

\author[M. A. Schmidt]{Marius A. Schmidt}
\address{Marius A. Schmidt \\ J.W. Goethe-Universit\"at Frankfurt, Germany.}
\email{mschmidt@math.uni-frankfurt.de}


 \date{\today}

\begin{abstract} 
It has been proved by Bovier \& Hartung [{\it Elect. J. Probab.} {\bf 19} (2014)] that the maximum of a variable-speed branching Brownian motion (BBM) in the weak correlation regime converges to a randomly shifted Gumbel distribution. The random shift is given by the almost sure limit of McKean's martingale, and captures the early evolution of the system. In the Bovier-Hartung extremal process, McKean's martingale thus plays a role which parallels that of the derivative martingale in the classical BBM. In this note, we provide an alternative interpretation of McKean's martingale in terms of a law of large numbers for {\it high-points} of BBM, i.e. particles which lie at a macroscopic distance from the edge. At such scales, 'McKean-like martingales' are naturally expected to arise in all models belonging to the BBM-universality class. 
\end{abstract}

\maketitle


\section{Introduction} \label{introduction and main result}
Over the last years, one has witnessed an explosion of activity in the study of the extremes of Branching Brownian Motion, BBM for short. The list of papers on the subject is way too long to be given here: below, we shall only mention those works which are indispensable for the discussion, and refer the reader to Bovier's monograph \cite{bovier} for an exhaustive overview of the literature. 

The classical, supercritical and  {\it time-homogeneous} BBM is constructed as follows.  A single particle performs standard Brownian motion $x(t)$, starting at $0$ at time $0$. After an exponential random time $T$ of mean one and independent of $x$, the particle splits into two (say) particles. The positions of these particles are independent Brownian motions starting at $x(T)$.
Each of these processes have the same law as the first Brownian particle. 
Thus, after a time $t>0$, there will be $n(t)$ particles  located at $x_1(t), \dots, x_{n(t)}(t)$, with $n(t)$ being the random number of offspring generated up to that time (note that $\E n(t)=e^t$). 

A fundamental link between BBM and partial differential equations 
was observed by McKean \cite{mckean}, who showed that the law of the maximal displacement of BBM solves the celebrated KPP-equation \cite{kpp}. Thanks to the cumulative works \cite{kpp, mckean, bramson, lalley_sellke} it is now known that the maximum of BBM weakly converges, upon recentering, to a random shift of the Gumbel distribution. More precisely, let 
\beq \label{centering_kpp}
m(t) \defi \sqrt{2} t - \frac{3}{2\sqrt{2}} \log t, \qquad M(t) \defi \max_{k\leq n(t)} x_k(t)-m(t)\ ,
\eeq
and consider the so-called {\it derivative martingale} 
\beq \label{defi_martingale}
Z(t) \defi \sum_{k\leq n(t)} \left(\sqrt{2}t - x_k(t) \right) \exp -\sqrt{2}\left(\sqrt{2}t - x_k(t) \right)\ .
\eeq
Leaning on the work of McKean  \cite{mckean} and Bramson \cite{bramson}, Lalley and Sellke \cite{lalley_sellke}
proved that
\beq \label{y_to_zero}
 \lim_{t\uparrow \infty} Z(t) = Z \text{ a.s.},
\eeq
with $Z$ a positive random variable, and that 
\beq 
\lim_{t \to \infty} \PP\left[ M(t) \leq x\right] = \E \exp - C Z e^{-\sqrt{2} x} \,.
\eeq
where $C>0$ is a numerical constant. Inspecting the proof of this result, one gathers that the derivative martingale captures the early evolution of the system. Perhaps a more intuitive interpretation of the derivative martingale has been given by Arguin, Bovier and Kistler \cite{abk_ergodicity} in the form of an 
ergodic theorem, to wit: 
\beq \bea \label{abk_e}
\lim_{t\uparrow \infty} \frac{1}{t} \int_0^t \1\{M(s) \leq x \} ~ ds = \exp\left( -C Z e^{-\sqrt{2} x}\right) \  \text{ almost surely.}
\eea \eeq
The derivative martingale may thus be seen as a {\it measure of success}, capturing the fraction of particles which reach maximal heights.

Also of interest are {\it time-inhomogeneous} BBMs. These have been first introduced by Derrida and Spohn \cite{derrida_spohn}, and are constructed as follows: one considers a BBM where, at time $s$, all particles move independently as Brownian motions with time-dependent variance 
\beq
\s^2(s) = \begin{cases}
\s_1^2 & 0 \leq s < t/2 \\
\s_2^2 & t/2 \leq s \leq t.
\end{cases}
\eeq
In the above, $\s_1, \s_2$ are (positive) parameters chosen in such a way that the total variance is normalised to unity, to wit: $\s_1^2 /2 + \s_2^2/2 =1$. (One may also consider $K>2$ distinct variance-regimes, but the qualitative picture does not change much, as long as $K$ remains finite). Denoting by $\hat n(s)$ the number of particles at time $s$, and by $\{\hat x_k(s), k\leq \hat n(s)\}$ their position, it has been proved by Fang and Zeitouni \cite{wizard_of_oz} that 
\beq
\max_{k \leq \hat n(t)} \hat x_k(t) = \begin{cases}
\sqrt{2} t - \frac{1}{2\sqrt{2}} \log(t) + O_\PP(1) & \text{if}\; \s_1 < \s_2 \\
\sqrt{2}(\s_1/2 + \s_2/2) - \frac{3}{2 \sqrt{2}}(\s_1+\s_2) \log(t) + O_\PP(1) & \text{if}\; \s_1 > \s_2,
\end{cases}
\eeq
The second case above, namely $\s_1 > \s_2$,  is easily understood: the maximum of the time-inhomogeneous process is given by the superposition of the relative maxima at time $t/2$ and the maxima of their offspring after a $t/2$-lifespan. 

The first case, $\s_1 < \s_2$, is arguably more interesting as it shows that the level of the maximum coincides with that of Derrida's REM \cite{derrida}: in spite of what may look as severe correlations between the Brownian particles, the extremes behave as if they were coming from a field of independent random variables.
This, however, is only true at the above level of precision, as the weak limit of the maximum does detect the underlying correlations. To formulate precisely, let us shorten 
\beq
m_{\text{REM}}(t) \defi \sqrt{2} t - \frac{1}{2\sqrt{2}} \log(t), \quad \hat M(t) \defi \max_{k\leq \hat n(t)} \hat x_k(t)- m_{\text{REM}}(t).
\eeq
Bovier and Hartung \cite{bovier_hartung} prove that in such a regime of "REM-collapse" ($\s_1 < \s_2$), 
\beq \label{bh}
\lim_{t\to \infty} \PP\left[ \hat M(t) \leq x\right] = \E \exp - \hat C \hat Z e^{-\sqrt{2} x} , 
\eeq
for a numerical constant $\hat C > 0$ and  $\hat Z$  a positive random variable. Much akin to the homogeneous BBM, the weak limit of the time-inhomogeneous process is thus given by a random shift of the Gumbel distribution. For our purposes, it is however crucial to emphasize that the random shift in case of REM-collapse is {\it not} given by the derivative martingale, but by the so-called {\it McKean's martingale}. To see how the latter comes about, recall the {\it time-homogeneous} BBM $\{x_k(t), k\leq n(t)\}$, and define  {\it McKean's martingale} by
\beq \label{mckean}
\hat Z(t) \defi \sum_{k \leq n(t)} \exp\left[- t (1+\s_1^2) + \sqrt{2} \s_1 x_k(t) \right].
\eeq
Bovier and Hartung \cite{bovier_hartung} show that this is, in fact, a square integrable martingale, {\it provided} that $\s_1 < 1$ strictly. It therefore converges almost surely to a well defined random variable whose law coincides with that of the $\hat Z$-random variable shifting the maximum \eqref{bh} of the {\it time-inhomogeneous} process\footnote{We emphasize tha the REM-collapse holds only for $\s_1 < \s_2$ ; given the normalization $\s_1^2/2+\s_2^2/2 = 1$, this is equivalent to $\s_1 < \sqrt{2}/2$. The square integrability of McKean's martingale  holds however for any $\s_1 < 1$. The  choice $\s_1 = 1$ corresponds to a boundary case where square integrability no longer holds. It has furthermore been proved by Lalley and Sellke \cite{lalley_sellke} that for $\s_1=1$ the limit of McKean's martingale  vanishes, in which case it is the derivative martingale that enters the picture for the weak limit of the maximum of (the time-homogeneous) BBM.}.

The analogy with the homogeneous BBM goes even further: an  inspection of the proof of the weak convergence \eqref{bh} shows that McKean's martingale captures, in fact, the early evolution of the system (remark, in particular, that McKean's martingale depends solely on $\s_1$). 

The purpose of these notes is to present yet another interpretation of McKean's martingale, somewhat close in spirit to the ergodic theorem \eqref{abk_e}. To formulate precisely, let 
\beq
\alpha \in (0, \sqrt{2}), \quad \Delta_\alpha \defi \sqrt{2}-\alpha, \quad \text{and} \quad 
Z_\alpha(t) \defi \sharp \left\{ k \leq n(t): x_k(t) \geq \Delta_\alpha t \right\}.
\eeq
The random variable $Z_\alpha(t)$ thus counts the "$\alpha$-high-points", those particles which lag behind the leader at time $t$ by a macroscopic distance $\alpha t$. Finally, consider the McKean's martingale 
\beq
Y_\alpha(t) \defi \sum_{k\leq n(t)} \exp\left[ - t \left( 1+  \frac{1}{2} \Delta_\alpha^2 \right)+ \Delta_\alpha x_k(t) \right]\,.
\eeq
Through the  matching  $\alpha \defi \sqrt{2}(1-\s_1)$, and by the aforementioned result of Bovier and Hartung we see that this is,  for $\alpha > 0$, a square integrable martingale whose limit
\beq \label{conv_bh}
\lim_{t \to \infty} Y_\alpha(t) =: Y_\alpha
\eeq
exists almost surely. Here is our main result.
\begin{teor} \label{main} (Strong law of large numbers) For any $0 < \alpha < \sqrt{2}$, and $Y_\alpha$ as in \eqref{conv_bh}, 
\beq
\lim_{t\to \infty} \frac{Z_\alpha(t)}{\E Z_{\alpha}(t)} = Y_\alpha\,, \quad \text{almost surely}.
\eeq
\end{teor}

According to the SLLN, the random shift entering the weak limit \eqref{bh} in the Bovier-Hartung extremal process thus captures the average number of {\it successful} particles. It is however noteworthy that the definition of {\it success} comes here with a twist: it pertains to those particles reaching\footnote{in a "distant past", as will become clear in the course of the proof.} heights which lie {\it macroscopically lower} than the level of the maximum; this should be contrasted with \eqref{abk_e}, where successful particles reach {\it extremal} heights. Half-jokingly, we may thus say that Lalley and Sellke's derivative martingale is way more elitist than McKean's martingale! \\

The proof of the SLLN follows the strategy of \cite{abk_ergodicity}. Precisely, for some $r =o(t)$ to be specified later, and $\mathcal F_r \defi \sigma\left( x_k(r), k \leq n(r) \right)$ the standard filtration of BBM, we first decompose telescopically 
\beq \label{tele}
\frac{Z_{\alpha}(t)}{\E Z_{\alpha}(t)} = \frac{\E\left[ Z_{\alpha}(t)  \mid  \mathcal F_r  \right]}{\E Z_{\alpha}(t)} + \frac{  Z_{\alpha}(t) -  \E\left[ Z_{\alpha}(t)  \mid  \mathcal F_r  \right] }{\E Z_{\alpha}(t)}. 
\eeq
The next Proposition will seamlessly follow from the strong Markov property of BBM and classical Gaussian estimates.

\begin{prop} \label{onset} (Onset of McKean's martingale) It holds: 
\beq \label{precise}
 \frac{\E\left[ Z_{\alpha}(t)  \mid  \mathcal F_r  \right]}{\E Z_{\alpha}(t)} = (1+o(1)) Y_\alpha(r), 
\eeq
almost surely. 
\end{prop}
By the above, and with our main theorem in mind, an important ingredient is therefore to prove that the second term on the r.h.s. of the telescopic decomposition \eqref{tele} yields an irrelevant contribution in the limit of large times. This is guaranteed by the following

\begin{prop} \label{van_corr} There exists $\kappa_\alpha >0$ such that 
\beq \label{rest_term}
\PP\left(\left|\frac{Z_{\alpha}(t) -  \E\left[ Z_{\alpha}(t)  \mid  \mathcal F_r  \right] }{\E Z_{\alpha}(t)}\right| >c \right)\leq (1+o(1))\frac{c+1}{c^2}e^{-\kappa_\alpha r} \,,
\eeq
for $r=o(t)$ as $t\to \infty$.
\end{prop}
A small remark concerning the conceptual picture behind Proposition \ref{van_corr} is perhaps at place. As mentioned,  the proof strategy and, in particular, the telescopic decomposition \eqref{tele}, are borrowed from \cite{abk_ergodicity}. In the latter paper, the counterpart of Proposition \ref{van_corr}, namely \cite[Theorem 3]{abk_ergodicity} holds thanks to a delicate decorrelation at specific timescales of the extremal particles of BBM which, in turns, is a consequence of the picture derived in \cite{abk_genealogy}. It is however unreasonable to expect here a similar decorrelation: $\alpha$-high-particles, namely those with $x_k(t) \geq\Delta_\alpha t$, are unlikely to come from (genealogically) distant ancestors. Indeed, quite the contrary is true:  a wealth of random variables contributing to $Z_\alpha(t)$ turn out to be strongly correlated, but these are {\it washed out}, in the limit of large times, by the exponentially large normalization  $\E Z_\alpha(t)$.  \\

Assuming Proposition \ref{onset} and \ref{van_corr}, our main theorem steadily follows.

\begin{proof}[Proof of Theorem \ref{main}]
We use the above telescopic decomposition 
\beq
\frac{Z_{\alpha}(t)}{\E Z_{\alpha}(t)} = \frac{\E\left[ Z_{\alpha}(t)  \mid  \mathcal F_r  \right]}{\E Z_{\alpha}(t)} + \frac{  Z_{\alpha}(t) -  \E\left[ Z_{\alpha}(t)  \mid  \mathcal F_r  \right] }{\E Z_{\alpha}(t)} =: (A) + (B)\,.
\eeq

By Proposition \ref{van_corr}, the $(B)$-term on the r.h.s. above vanishes in probability as $t \uparrow \infty$ first, and $r\uparrow \infty$ next. In fact, we  may lift this to an almost sure statement in the single limit $t\uparrow \infty$: simply choose $r = r(t) \to \infty$ such that the r.h.s of \eqref{rest_term} becomes integrable (any choice of the form $r = (\log t)^\epsilon$ with $\epsilon> 1$ will do) and appeal to the Borel-Cantelli Lemma together with standard approximation arguments (for the latter, see e.g. \cite{abk_genealogy}). But for such choice $r = r(t)$, and the almost sure convergence of the McKean's martingale established by Bovier-Hartung, the $(A)$-term will converge, almost surely, to $Y_\alpha$. This settles the proof of the SLLN.
\end{proof}

The rest of the paper is devoted to the proofs of Proposition \ref{onset} and \ref{van_corr}. Before giving the details, we conclude this section with the following \\

\noindent {\bf Conjecture}.  {\it A SLLN as in Theorem \ref{main} holds true, \emph{mutatis mutandis}, in all models belonging to the BBM-universality class, such as the 2-dim Gaussian free field \cite{bdg, bdz, b_d_g, bl, bl_two, bl_three}, the  2-dim cover times \cite{dprz, bk, brz}, the characteristic polynomials of random unitary matrices \cite{abb, cnm, pz}, and the extreme values of the Riemann zeta function on the critical line \cite{abh, najnudel, abbrs, at}. In particular, we expect that an \emph{approximate McKean's martingale} will capture in all such models the almost sure limit of the normalized number of high-points. (What stands behind this wording becomes, of course, model-dependent).} \\

\noindent{\bf Acknowledgments.} It is a pleasure to thank David Belius for raising the question of high-points in the 2dim GFF, which lead to the writing of this paper. 

\section{Proofs}

\subsection{Some preliminaries, and Onset of McKean's martingale}
We will make constant use of some classical Gaussian tail-estimates:

\begin{lem} \label{tail} 
Let $X \sim \mathcal N(0, \s^2)$ be centered Gaussian random variables. Then 
\beq
\PP\left[ X> a\right] = \left(2\pi \right)^{-1/2} (\s/a) \exp\left[ - \frac{a^2}{2 \s^2} \right]\left( 1+ O\left(\s^2/a^2\right)\right)\, \qquad (a/\s \to \infty)\,,
\eeq
with the r.h.s. above \emph{without} error term being an upper bound valid for any $a>0$.
\end{lem}
The following is also elementary.
\begin{lem} \label{asymptotics}
\beq
\E Z_\alpha(t) \sim \left(\Delta_\alpha \sqrt{2\pi}\right)^{-1} \exp\left[ \left(1-\Delta_\alpha^2/2\right)t- \frac{1}{2} \log(t) \right].
\eeq
\end{lem}
(Here and throughout, we use $f(t) \sim g(t)$ if the ratio converges, as $t \to \infty$, to one).  In order to exploit the strong Markov property of BBM, for $t$ and $r$ as in Proposition \ref{onset} , we re-label particles at time $t$ according to their ancestor at time $r$:
\beq
\left(x_k(t)\right)_{k\leq n(t)} = \left( x_{i}(r) + x_{i,j}(t-r) \right)_{i \leq n(r), j \leq n_i(t-r)}.
\eeq
Precisely: $x_{i}(r) $ is the position of the $i$-th particle at time $r$, $n(r)$ is the number of such particles,  $n_{i}(t-r)$ stands for the number of offspring such particle has produced in the timespan $t-r$, and  finally $x_{i,j}(t-r)$ denotes the displacement of the $j$-th offspring of particle $i$ from its starting position $x_{i}(r)$.

\begin{proof}[Proof of Proposition \ref{onset} (Onset of McKean's martingale)] 
\beq \bea 
& \E\left[ Z_\alpha(t) \mid \mathcal F_r \right] = \E\left[  \sum_{k \leq n(t)} {\boldsymbol 1}\left\{ x_k(t) \geq \Delta_\alpha t \right\}  \Big| \mathcal F_r \right]  = \\
& = \E\left[  \sum_{i\leq n(r)} \sum_{j \leq n_i(t-r)} {\boldsymbol 1}\left\{ x_{i,j}(t-r) \geq \Delta_\alpha (t-r) -  (x_{i}(r) -\Delta_\alpha r )  \right\}  \Big| \mathcal F_r \right] \\
& = \sum_{i \leq n(r)} e^{t-r} \PP\left[ x_1(t-r) \geq \Delta_\alpha (t-r) -  (x_{i}(r) -\Delta_\alpha r ) \Big| \mathcal F_r \right] \\
& \sim  \left(\Delta_\alpha \sqrt{2\pi}\right)^{-1} \exp\left[ \left(1-\Delta_\alpha^2/2\right)t- \frac{1}{2} \log(t) \right] Y_\alpha(r), \quad \mbox{a.s.}.
\eea \eeq
The last step by combining Lemma \ref{tail} with the fact that $\left(x_{k}(r)-\Delta_\alpha r\right)/\left(t-r\right) \rightarrow 0$  almost surely for $r=o(t)$, see e.g. \cite{hu_shi} (where, in fact, control at even finer levels is provided). The
claim then follows by Lemma \ref{asymptotics}.
\end{proof}

\subsection{Vanishing correlations, via second moment}
We will prove Proposition \ref{van_corr} by a second moment estimate. For these computations to go through, we need however to localize paths of contributing particles. (This approach is by now classical in the BBM-field, see  for instance \cite{abk_ergodicity} for a closely related setting).  As for the localization, let $\varepsilon>0$  and consider 
\beq
Z^{>}_\alpha(t) \defi \sharp \left\{ k \leq n(t): x_k(t) \geq \Delta_\alpha t, \exists s\in[r,t]: x_k(s)> (\Delta_\alpha +\varepsilon)s \right\}, 
\eeq
This random variable thus counts paths which overshoot at some point (in time) the straight line connecting $0$ to $(\Delta_\alpha +\varepsilon)t$. As it turns out, such particles do not contribute, upon  $\E[Z_\alpha (t)]$-normalization,  to the $\alpha$-high-points. Here is the precise statement. 

\begin{lem} (Paths-localization) \label{barrier}
For $r= o(t)$, $r,t$ both sufficiently large, it holds:
\beq
\PP\Big(  Z^{>}_\alpha(t) \geq c \E Z_\alpha(t) \Big) \leq \frac{1}{c}  \exp\left( - r \frac{\varepsilon^2}{4}  \right)\,, \eeq
and
\beq
 \PP\Big( \E[ Z^{>}_\alpha(t)|\mathcal{F}_r] \geq c \E Z_\alpha(t) \Big) \leq \frac{1}{c}  \exp\left( - r \frac{\varepsilon^2}{4} \right)\,. 
\eeq
\end{lem}

\begin{proof}
\beq
\E Z^{>}_\alpha(t) = e^t \int\limits_{\Delta_\alpha t}^\infty \PP(x_1(t) \in dy) \PP(\exists s\in[r,t]: x_1(s) > (\Delta_\alpha + \varepsilon)  s   | x_1(t) = y) 
\eeq
\beq \label{brow_bridge}
=e^t \int\limits_{\Delta_\alpha t}^\infty \PP(x_1(t) \in dy) \PP(\exists s\in[r,t]: b(s) > (\Delta_\alpha + \varepsilon - \frac{y}{t})  s ) ,
\eeq
where $b(s) \defi x_1(s)-\frac{s}{t}x_1(t)$ is a Brownian bridge of length $t$. Consider the line $l$ from $(0,\varepsilon r/2)$ to $(t,(\Delta_\alpha + \varepsilon/2)t - y)$. One easily checks that $l(s) \leq (\Delta_\alpha + \varepsilon - \frac{y}{t})  s$ for all $s\in[r,t]$. Hence  the probability involving the Brownian brige is at most 
\beq 
\PP(\exists s\in[0,t]: b(s) > l(s) ) = \exp\left( -2  \frac{l(0)l(t)}{t}\right),
\eeq  
by a well-known formula (see e.g. \cite{scheike}). Using this, \eqref{brow_bridge} is  therefore at most
\beq \bea
& e^t \int\limits_{\Delta_\alpha t}^\infty  \frac{1}{\sqrt{2\pi t}}\exp\left( - \frac{y^2}{2t} - \varepsilon r \left( \Delta_\alpha +\frac{\varepsilon}{2} - \frac{y}{t}  \right)\right)dy \\
& \qquad =  \exp\left( t - r \left( \Delta_\alpha \varepsilon+\frac{\varepsilon^2}{2} \right) + \frac{\varepsilon^2 r^2}{2t}\right)  \int\limits_{\Delta_\alpha t}^\infty  \frac{1}{\sqrt{2\pi t}}\exp\left( - \frac{(y-\varepsilon r)^2 }{2t}  \right)dy,
\eea \eeq
which is, by Lemma \ref{tail}, 
\beq \bea
& \sim \frac{\sqrt{t}}{(\Delta_\alpha t -\varepsilon r)\sqrt{2\pi}} \exp\left( t - r \left( \Delta_\alpha \varepsilon+\frac{\varepsilon^2}{2} \right) + \frac{\varepsilon^2 r^2}{2t} - \frac{(\Delta_\alpha t -\varepsilon r)^2 }{2t}\right) \\
& \sim  \E[Z_\alpha(t)] \exp\left( - r \frac{\varepsilon^2}{2}  + o(r) \right),
\eea \eeq
the second asymptotical equivalence by Lemma \ref{asymptotics}. The claim of Lemma \ref{barrier} thus follows from Markov inequality. 
\end{proof}

As mentioned, we will prove Proposition \ref{van_corr} by means of a (truncated) second moment computation. The following is the key estimate.

\begin{lem} (Pair-counting) \label{twopoint}
Let $I_{i,j}$ be the the Indicator of the event  that the $j$-th offspring at time $t$ of particle $i$ at time $r$ contributes to $Z_\alpha^{\leq}(t)$, i.e the event
\beq
\{x_i(r) + x_{i,j}(t-r) \geq \Delta_\alpha t, \forall s\in[r,t]: x_{i}(r)+ x_{i,j}(s-r) \leq (\Delta_\alpha+ \varepsilon ) s \}
\eeq
for $i\leq n(r)$ and $ j\leq n_{i}(r)$. Then, for any $\alpha\in (0,\sqrt{2})$ there exists $\vare_\alpha$ and $\kappa_\alpha,  r(\alpha)>0$ such that  
\beq
\E\left[\sum\limits_{i=1}^{n(r)} \sum\limits_{j\neq j'=1}^{ n_i(t-r)} I_{i,j}I_{i,j'}  \right] \leq  (1+o(1))\E[Z_\alpha(t)]^2 e^{-\kappa_\alpha r}, \quad \mbox{ as } t\rightarrow\infty\,,
\eeq
for $r > r_\alpha$.
\end{lem}
\begin{proof}
Denoting by $\varphi_\gamma$ the density of a centered Gaussian of variance $\gamma$, and with $x$ a standard Brownian motion, it holds that
\beq \bea \label{sawyer}
& \frac{1}{2}\E\left[\sum\limits_{i=1}^{n(r)} \sum\limits_{j\neq j'=1}^{ n_i(t-r)} I_{i,j}I_{i,j'}  \right] = \\
& \qquad =  \int\limits_{r}^{t} e^{2t-\gamma} \int\limits_{-\infty}^{(\Delta_\alpha +\varepsilon) \gamma} \varphi_\gamma(y)\PP(\forall s\in [r,\gamma]:x(s) \leq (\Delta_\alpha +\varepsilon)  s| x(\gamma) = y )\times \\
& \hspace{1cm} \times  \PP\left(y+ x(t-\gamma) \geq \Delta_\alpha t, \forall s\in [\gamma,t]: y+x(s-\gamma)\leq (\Delta_\alpha +\varepsilon ) s \right)^2 dy d\gamma\,.
\eea \eeq
(See e.g.  \cite{sawyer} for a rigorous derivation of a similar "two-point formula"). Dropping the path-constraint appearing in the integrand (yet {\it keeping} those in the domain of integration), \eqref{sawyer} is at most
\beq \bea \label{sawyer_two}
& \int\limits_{r}^{t} e^{2t-\gamma} \int\limits_{-\infty}^{(\Delta_\alpha +\varepsilon) \gamma} \varphi_\gamma(y)\PP\left( x(t-\gamma) \geq \Delta_\alpha t -y\right)^2 dy d\gamma\ \\
& =  \int\limits_{r}^{t} e^{2t-\gamma} \int\limits_{-\infty}^{(\Delta_\alpha +\varepsilon) \gamma} \mathbbm{1}_{\left\{y\geq (\Delta_\alpha-\varepsilon) t \right\}} (\cdot) dy d \gamma + \int\limits_{r}^{t} e^{2t-\gamma} \int\limits_{-\infty}^{(\Delta_\alpha +\varepsilon) \gamma} \mathbbm{1}_{\left\{y< (\Delta_\alpha-\varepsilon) t\right\}}  (\cdot) dy d \gamma \,, 
\eea \eeq
by distinguishing whether at time of splitting particles are above (respectively below) a threshold which is slightly below the target. 

As for the first scenario, we clearly have 
\beq \bea \label{sawyer_3}
\int\limits_{r}^{t} e^{2t-\gamma} \int\limits_{-\infty}^{(\Delta_\alpha +\varepsilon) \gamma} \mathbbm{1}_{\left\{y\geq (\Delta_\alpha-\varepsilon) t \right\}} (\cdot) dy d \gamma \leq \int\limits_{(1-2\varepsilon/\Delta_\alpha)t }^{t} e^{2t-\gamma} \int\limits_{\Delta_\alpha t- \varepsilon t}^{(\Delta_\alpha+\varepsilon) \gamma } \varphi_\gamma(y) dy d\gamma\,,
\eea \eeq
with the r.h.s. of \eqref{sawyer_3} being at most 
\beq \label{rough_est1}
\frac{4\varepsilon^2 t^2}{\Delta_\alpha\sqrt{2\pi t}}\exp\left( (1+2\varepsilon/\Delta_\alpha)t -\frac{(\Delta_\alpha t- \varepsilon t)^2}{2t}\right)  \leq \exp\left( (1- \Delta_\alpha^2/2  +\varepsilon (2/\Delta_\alpha+\Delta_\alpha))t \right), 
\eeq
for  $t$ large enough. Since $1- \Delta_\alpha^2/2 +\varepsilon (2/\Delta_\alpha+\Delta_\alpha)< 2- \Delta_\alpha^2$ for some $\varepsilon = \vare_\alpha$, it follows that \eqref{rough_est1} grows at most exponentially (in $t$) with rate smaller than  $2- \Delta_\alpha^2$, and therefore yields a negligible contribution. 

It thus remains to analyse the second scenario in \eqref{sawyer_two}. By Lemma \ref{tail} and shortening $\zeta \defi \min \{(\Delta_\alpha +\varepsilon) \gamma, \Delta_\alpha t- \varepsilon t\} $, we have that
\beq \label{sawyer_four} \bea
& \int\limits_{r}^{t} e^{2t-\gamma} \int\limits_{-\infty}^{(\Delta_\alpha +\varepsilon) \gamma} \mathbbm{1}_{\left\{y< (\Delta_\alpha-\varepsilon) t\right\}}  (\cdot) dy d \gamma \\
& \leq \int\limits_{r}^{t} e^{2t-\gamma} \int\limits_{-\infty}^{\zeta} \frac{t-\gamma}{2\pi \sqrt{\gamma}(\Delta_\alpha t - y)^2} \exp\left(-\frac{y^2}{2\gamma}-\frac{(\Delta_\alpha t - y)^2}{(t-\gamma)}\right) dy d\gamma \\
& = \int\limits_{r}^{t} e^{2t-\gamma} \int\limits_{-\infty}^{\zeta} \frac{t-\gamma}{2\pi \sqrt{\gamma}(\Delta_\alpha t - y)^2} \exp\left(-\frac{\Delta_\alpha^2 t^2}{t+\gamma}- \frac{(y-\frac{2\gamma \Delta_\alpha t}{t+\gamma})^2}{2\gamma (t-\gamma)/(t+\gamma)}\right) dy d\gamma .
\eea \eeq
Since $\Delta_\alpha t - y \geq \varepsilon t$ on the entire domain of integration, and rearranging, \eqref{sawyer_four} is at most
\beq \label{sawyer_five}
\int\limits_{r}^{t} \exp\left(2t-\gamma-\frac{\Delta_\alpha^2 t^2}{t+\gamma}\right) \frac{\sqrt{t^2-\gamma^2}}{\varepsilon^2 t^2 \gamma\sqrt{2\pi}}\PP\left( x\left(\frac{\gamma(t-\gamma)}{t+\gamma}\right) < \zeta - \frac{2\gamma \Delta_\alpha t}{t+\gamma} \right) d\gamma.
\eeq

We split \eqref{sawyer_five} again into two regions: the first concerns $\gamma> (1-\delta) t$, with $\delta \defi \frac{3\varepsilon}{\Delta_\alpha +\varepsilon}$. In this case, estimating the probability by one and the remaining integrand by a rough bound on its maximum yields a contribution of at most 
\beq \bea \label{rough_est2}
& \int\limits_{(1-\delta)t}^{t} \exp\left(2t-\gamma-\frac{\Delta_\alpha^2 t^2}{t+\gamma}\right) \frac{\sqrt{t^2-\gamma^2}}{\varepsilon^2 t^2 \gamma\sqrt{2\pi}}\PP\left( x\left(\frac{\gamma(t-\gamma)}{t+\gamma}\right) < \zeta - \frac{2\gamma \Delta_\alpha t}{t+\gamma} \right)d\gamma \\
& \hspace{2cm}\leq \delta t \exp\left((1-\frac{\Delta_\alpha^2 }{2}+\delta) t  \right) \frac{\sqrt{t^2-\gamma^2}}{\varepsilon^2 t^2 \gamma\sqrt{2\pi}} \leq  \exp\left((1-\frac{\Delta_\alpha^2 }{2}+2\delta) t  \right)\,,
\eea \eeq
for $t$ large enough. Therefore, for $\delta$ or equivalently $\varepsilon$ small enough (depending on $\alpha$ only), this term is also negligeable. 

The second case in \eqref{sawyer_five} pertains to $\gamma< (1-\delta) t$:  in this region, and due to the choice of $\delta$,  we have 
\beq 
\zeta = (\Delta_\alpha +\varepsilon) \gamma \Longrightarrow \zeta - \frac{2\gamma \Delta_\alpha t}{t+\gamma} =  \gamma \left(  \varepsilon -\Delta_\alpha \frac{ t-\gamma}{t+\gamma}\right)<0,
\eeq 
(the last estimate again due to the choice of $\delta$, and for sufficiently small $\varepsilon$ depending on $\alpha$ only).  This, together with Lemma \ref{tail}, implies that
\beq \bea \label{aaah}
& \int\limits_{r}^{(1-\delta)t} \exp\left(2t-\gamma-\frac{\Delta_\alpha^2 t^2}{t+\gamma}\right) \frac{\sqrt{t^2-\gamma^2}}{\varepsilon^2 t^2 \gamma\sqrt{2\pi}}\PP\left( x\left(\frac{\gamma(t-\gamma)}{t+\gamma}\right) < \zeta - \frac{2\gamma \Delta_\alpha t}{t+\gamma} \right) d\gamma \\
& \hspace{2cm} \leq  \int\limits_{r}^{(1-\delta)t} \exp\left(2t-\gamma-\frac{\Delta_\alpha^2 t^2}{t+\gamma}-\frac{\gamma \left(\Delta_\alpha (t-\gamma) -\varepsilon(t+\gamma)\right)^2}{2(t^2-\gamma^2)}\right) \times \\
& \hspace{7cm} \times \frac{t^2-\gamma^2}{\left(\Delta_\alpha (t-\gamma) -\varepsilon(t+\gamma)\right)\varepsilon^2 t^2 \gamma\sqrt{2\pi \gamma}}
 d\gamma\,.
\eea \eeq
Using that $-\frac{\gamma\varepsilon^2(t+\gamma)^2}{2(t^2-\gamma^2)}< 0$, and by simple algebra, \eqref{aaah} is at most
\beq \bea \label{last_calc}
& \int\limits_{r}^{(1-\delta)t} \exp\left((2-\Delta^2_\alpha)t+(\frac{\Delta_\alpha^2}{2}-1+\Delta_\alpha \varepsilon)\gamma\right) \frac{2}{\Delta_\alpha(1-\frac{3\varepsilon}{\Delta_\alpha \delta})\varepsilon^2 \gamma \sqrt{\gamma}2\pi t}
 d\gamma \\
& \sim \E[Z_\alpha(t)]^2 \frac{2 \Delta_\alpha}{(1-\frac{3\varepsilon}{\Delta_\alpha \delta})\varepsilon^2} \int\limits_{r}^{(1-\delta)t} \gamma^{-3/2}\exp\left((\frac{\Delta_\alpha^2}{2}-1+\Delta_\alpha \varepsilon)\gamma\right) d\gamma \,,
\eea \eeq
by Lemma \ref{asymptotics}. But for $\varepsilon$ sufficiently small $\Delta_\alpha^2/2-1+\Delta_\alpha \varepsilon <0$, hence the integral in \eqref{last_calc} vanishes exponentially fast in $r$.  In other words, \eqref{last_calc} can be bounded by $\E[Z_\alpha(t)]^2 \exp\left(-\kappa_\alpha r\right)$ for any $\kappa_\alpha < 1-\Delta_\alpha^2/2$ and $r$ large enough. Combining this with the fact that the two error-terms \eqref{rough_est1} and \eqref{rough_est2} are of negligeable size compared to $\E[Z_\alpha(t)]^2$ finishes the proof.
\end{proof}

We can now move to the

\begin{proof}[Proof of Proposition \ref{van_corr}]
By the paths-localization from Lemma \ref{barrier}, and for $t$  large enough,
\beq \bea \label{almost_done}
\PP\left(\left|\frac{Z_\alpha(t)- \E\left[Z_\alpha(t)|\mathcal{F}_r \right]}{\E\left[Z_\alpha(t) \right]}\right|\geq c\right) & \leq
\PP\left( \left|\frac{Z^{\leq}_\alpha(t)- \E\left[Z^{\leq}_\alpha(t)|\mathcal{F}_r \right]}{\E\left[Z_\alpha(t) \right]}\right|\geq c/2 \right) + \frac{8}{c}  \exp\left( - r \frac{\varepsilon^2}{4}  \right)  \\
&  \leq \frac{4\E\left[ \left(Z^{\leq}_\alpha(t)- \E\left[Z^{\leq}_\alpha(t)|\mathcal{F}_r \right]\right)^2 \right]}{c^2\E\left[Z_\alpha(t) \right]^2 }+\frac{8}{c}  \exp\left( - r \frac{\varepsilon^2}{4}  \right) ,
\eea \eeq
the last step by Markov inequality. We thus need sufficiently good (upper) bounds for
\beq
\E\left[ \left(Z^{\leq}_\alpha(t)- \E\left[Z^{\leq}_\alpha(t)|\mathcal{F}_r \right]\right)^2 \right]=\E\left[ \E\left[(Z_\alpha^{\leq}(t))^2 |\mathcal{F}_r \right]- \E\left[Z_\alpha^{\leq}(t)|\mathcal{F}_r \right]^2\right].
\eeq
Adopting the notation of Lemma \ref{twopoint}, 
\beq
Z_\alpha^{\leq}(t) = \sum\limits_{i\leq n(r)} \sum\limits_{j\leq n_i(t-r)} I_{i,j}, 
\eeq
in which case
\beq \bea \label{above_form}
& \E\left[(Z_\alpha^{\leq}(t))^2 |\mathcal{F}_r \right]- \E\left[Z_\alpha^{\leq}(t)|\mathcal{F}_r \right]^2 = \\
& = \qquad \sum\limits_{i,i'=1}^{n(r)} \E\left[ \left. \left(\sum\limits_{j=1}^{n_i(t-r)} I_{i,j} \right)\left( \sum\limits_{j'=1}^{n_{i'}(t-r)} I_{i',j'}\right)\right|\mathcal{F}_r\right] \\
& \hspace{5cm} -\E\left[ \left.\sum\limits_{j=1}^{ n_i(t-r)} I_{i,j} \right|\mathcal{F}_r\right] \E\left[ \left.\sum\limits_{j'=1}^{n_{i'}(t-r)} I_{i',j'}\right|\mathcal{F}_r\right].
\eea \eeq
In the above, and for $i\neq i'$, particles $(i,j)$ and $(i',j')$  have branched off before time $r$: they 
are thus independent, conditionally upon $\mathcal{F}_r$. This leads to a perfect cancellation of all terms $i\neq i'$, and reduces the above formula to  
\beq
\E\left[(Z_\alpha^{\leq}(t))^2 |\mathcal{F}_r \right]- \E\left[Z_\alpha^{\leq}(t)|\mathcal{F}_r \right]^2= 
\sum\limits_{i=1}^{n(r)} \E\left[ \left. \left(\sum\limits_{j=1}^{ n_i(t-r)} I_{i,j} \right)^2 \right|\mathcal{F}_r\right]-\E\left[ \left.\sum\limits_{j=1}^{ n_i(t-r)} I_{i,j} \right|\mathcal{F}_r\right]^2 \,.
\eeq
Dropping the second term, and taking expectations, we thus obtain 
\beq \label{really_last}
\E\left[ \left(Z^{\leq}_\alpha(t)- \E\left[Z^{\leq}_\alpha(t)|\mathcal{F}_r \right]\right)^2 \right] \leq \E\left[\sum\limits_{i=1}^{n(r)} \sum\limits_{j,j'=1}^{ n_i(t-r)} I_{i,j}I_{i,j'}  \right].
\eeq
Collecting the terms $j=j'$ yields $Z_\alpha^{\leq}(t)$, while all other terms sum over all unordered pairs of particles that have split after time $r$. Choosing $\varepsilon= \varepsilon_\alpha$ small enough and $r=o(t)$ large enough, by Lemma \ref{twopoint} there exists $\kappa_\alpha >0$ such that 
\beq \label{final_est}
\eqref{really_last} \leq \E[Z_\alpha^{\leq}(t)] + (1+o(1))\E[Z_\alpha(t)]^2 e^{-\kappa_\alpha r} = (1+o(1))\E[Z_\alpha(t)]^2 e^{-\kappa_\alpha r}.
\eeq
The claim thus follows by plugging \eqref{final_est}  into \eqref{almost_done}. 
\end{proof}

\end{document}